\documentclass[11pt,a4paper]{amsart}

\usepackage{amssymb, amsmath, amsthm, amscd}

\usepackage{enumerate}

\usepackage{mathrsfs}

\usepackage{tikz-cd}

\theoremstyle{plain} \newtheorem{theorem}{Theorem}[section]
\theoremstyle{plain} 
\theoremstyle{plain} 
\theoremstyle{plain}\newtheorem{lemma}[theorem]{Lemma}
\theoremstyle{definition} 
\theoremstyle{definition}
\theoremstyle{remark}\newtheorem{remark}[theorem]{Remark}
\theoremstyle{definition}\newtheorem{hypothesis}[theorem]{Hypothesis}
\theoremstyle{remark}
\theoremstyle{definition}
\theoremstyle{definition}

\newcommand{\C}{{\mathbb{C}}}
\newcommand{\R}{{\mathbb{R}}}

\newcommand{\I}{{\mathscr{I}_{\ell}}}

\numberwithin{equation}{section}

\usepackage{hyperref}
\hypersetup{colorlinks = true, linkcolor=blue}

\title[Moduli Stack of invariant complex structures on $G$]{A geometric compactification of the moduli stack of left invariant complex structures on a Lie group}
\author{Laurent Meersseman}
\date{\today}

\subjclass{32G05, 
	53C15, 
	58D27, 
	14D23
}

\address{Laurent Meersseman\\
	Univ Angers, CNRS, LAREMA, SFR MATHSTIC\\
	F-49000 Angers, France\\ laurent.meersseman@univ-angers.fr}

\thanks{The author benefits from the support of the French government “Investissements d’Avenir” program integrated to France 2030, bearing the following reference ANR-11-LABX-0020-01.}

\begin{document}
	\begin{abstract}
		We describe a geometric compactification of the moduli stack of left invariant complex structures on a fixed real Lie group or a fixed quotient. The extra points are CR structures transverse to a real foliation. 
	\end{abstract}
	
	\maketitle

	\section{Introduction}
	\label{intro}

Left-invariant complex structures on Lie groups and their quotients form a very specific but also very rich class of (mostly) non-Kähler compact complex manifolds, as shown for example by the following two facts:
\begin{enumerate}[---]
	\item Such structures exist on {\slshape any} compact real Lie group of even dimension \cite{Samelson}.
	\item Small deformations of complex parallelizable nilmanifolds, that is quotients of a complex nilpotent Lie group by a cocompact lattice acting on the right, are in general no longer parallelizable but are still given by left invariant structures on the Lie group \cite{Rollenske1}\footnote{ Observe however that small deformations of left invariant structures are in general not left invariant, cf. Hopf surfaces or the beautiful deformations of \cite{Ghys}.}. Their Kuranishi space can be very singular and also not reduced \cite{Rollenske2}. 
\end{enumerate}
In this short note, we describe a geometric compactification of the moduli stack of left invariant complex structures on a fixed real Lie group or a fixed quotient. The starting point is the following easy observation detailed in \S \ref{secgrassmann}.
Left-invariance implies that such complex structures are completely determined by their value at a single point and can be described by an endomorphism of the Lie algebra. Looking at the eigenspaces of this endomorphism over $\C$ allows to identify the set of left invariant structures $\I$ with an open set of a projective variety $\mathbb{V}$ living in the grassmannian of half-dimensional planes of the complexified Lie algebra. Hence $\mathbb{V}$ is a natural compactification of $\I$.

The problem is now to understand which objects correspond to the extra points and which type of moduli stack is associated to them.

We describe geometrically in \S \ref{sectransverseCRfoliations} points in $\mathbb{V}$ that are not complex structures as real foliations with a transverse CR structure or equivalently as holomorphic foliations on the complexified Lie group. We also derive a notion of family of such structures giving rise to an associated moduli stack. We obtain in \S \ref{secmodulistack} a corresponding compactification of the moduli stack of left invariant structures under some mild hypothesis. Examples are treated in \S \ref{secexamples}.

The main interesting point of this construction is that the compactification is not obtained by considering flat families instead of smooth ones, i.e. not obtained by adding singular geometric objects to the moduli stack, but rather by considering families of $C^\infty$ manifolds endowed with more general geometric structures.

It is also worth noticing that in this toy example, there is a nice dictionary between complex structures on the group and foliations on a complexification of the group, with Newlander-Nirenberg Theorem responding to Frobenius Theorem. Unfortunately, this seems to be one of the sole class of examples where such an effective dictionary can be established.

I would like to thank Ernesto Lupercio for pointing out the relation with \cite{KE}, cf . Remark \ref{rkKontsevich}, opening the way to future developments.

\section{A natural compactification of the set of left invariant complex structures}
\label{secgrassmann}
Let $G$ be a connected real Lie group of dimension $2n$. A {\slshape left invariant almost-complex structure} on $G$ is an almost-complex operator $J$ on the tangent bundle $TG$ that is invariant by pull-back by any left translation. Equivalently $J$ maps left invariant vector fields to left invariant vector fields hence induces a linear endomorphism of the Lie algebra $\mathfrak{G}$ of $G$. Left-invariant almost-complex structures on $G$ are in $1:1$-correspondence with linear endomorphisms of $\mathfrak{G}$ whose square is $-Id$. Recall that we denote by $\I$ the set of left invariant complex structures on $G$. It is endowed with the topology of $C^\infty$ convergence of operators $J$ on compact subsets of $G$.

We denote by the same letter $J$ the operator on $TG$ and the corresponding endomorphism of $\mathfrak{G}$.

\begin{remark}
	\label{rkoriented}
	We do not ask our structures to be compatible with some fixed orientation on $G$ or equivalently on $\mathfrak{G}$ as it is usually the case in moduli theory. The reason for this non standard choice is that we need to consider the space of {\slshape all} complex strutures to have a complex projective compactification. This will become clear in \S \ref{subsectori}.
\end{remark} 

The endomorphism $J$ is diagonalizable over $\C$ with eigenvalues $+i$ and $-i$. This gives a decomposition $\mathfrak{G}\otimes_\R \C=T^{1,0}\oplus T^{0,1}$. Extend linearly the Lie bracket to $\mathfrak{G}_\C:=\mathfrak{G}\otimes_\R \C$. Since the Lie bracket on $\mathfrak{G}$ corresponds to the Lie bracket on left invariant vector fields, a left invariant structure $J$ is integrable, i.e. defines a structure of complex manifold on $G$ through Newlander-Nirenberg Theorem if and only the subspace $T^{0,1}$ is involutive, i.e. $[T^{0,1},T^{0,1}]\subset T^{0,1}$. Note that the integrable left invariant almost-complex structures are exactly the complex structures on $G$ such that all left translations are automorphisms.

As a consequence, let $\text{Gr}_{n}(\mathfrak{G}_\C)$ be the grassmaniann of $n$-planes of the complex $2n$-dimensional vector space $\mathfrak{G}_\C$. Define
\begin{equation}
	\label{V}
	\mathbb{V}:=\{T\in\text{Gr}_{n}(\mathfrak{G}_ \C)\mid [T,T]\subset T\}
\end{equation}
Then, we have
\begin{lemma}
	\label{lemmagrass}
	The set $\mathbb{V}$ is a subvariety of $\text{\rm Gr}_{n}(\mathfrak{G}_\C)$ and the mapping
	\begin{equation}
		\label{Imapping}
		J\in\I\longmapsto T^{0,1}\in \text{\rm Gr}_{n}(\mathfrak{G}_ \C)
	\end{equation}
	identifies the space $\I$ with the open subset
	\begin{equation}
		\label{I}
		\{T\in\mathbb{V}\mid T\cap \bar T=\{0\}\}
	\end{equation}
	of the projective variety $\mathbb{V}$.
\end{lemma}

\begin{proof}
	Consider the vector space $\Lambda^n(\mathfrak{G}_\C)$ of $n$-vectors of $\mathfrak{G}_\C$ and endow it with the Niejenhuis-Schouten bracket extending the Lie bracket on $\mathfrak{G}_\C$. Let $\mathscr{P}\subset \Lambda^n(\mathfrak{G}_\C)$ be the affine cone of pure $n$-vectors. The map
	\begin{equation*}
		T\in\mathscr{P}\longmapsto ([T,T],T)\in \Lambda^n(\mathfrak{G}_\C)\times\Lambda^n(\mathfrak{G}_\C)
	\end{equation*}
	sends the cone $\mathbb{A}$ of involutive pure $n$-vectors to the analytic subspace of couples of collinear vectors in $\Lambda^n(\mathfrak{G}_\C)\times\Lambda^n(\mathfrak{G}_\C)$. Hence $\mathbb{A}$ is an affine cone in  $\mathscr{P}\subset \Lambda^n(\mathfrak{G}_\C)$. Therefore the projectivization of $\mathbb{A}\setminus\{0\}$, which is nothing else than $\mathbb{V}$, is a subvariety of the projectivization of $\mathscr{P}\setminus\{0\}$, that is of $\text{Gr}_{n}(\mathfrak{G}_\C)$.
	
	We already argued that integrability forces the image of \eqref{Imapping} to land in $\mathbb{V}$. Now if $T=T^{0,1}$ for some complex structure, then $\bar{T}=T^{1,0}$ and these two subspaces must be in direct sum. The inverse of \eqref{Imapping} is the operator $J$ given by multiplication by $-i$ on $T$ and by multiplication by $+i$ on $\bar T$. Finally, the condition $T\cap \bar T=\{0\}$ is open in $\text{\rm Gr}_{n}(\mathfrak{G}_ \C)$ and thus in $\mathbb{V}$.
\end{proof}

\section{Transversely CR real foliations}
\label{sectransverseCRfoliations}
Given $X$ a $C^\infty$ manifold, a real foliation with transverse structure is a $C^\infty$ foliation with a geometric structure (e.g. complex structure, riemanniann, Kähler structure) on its normal bundle that is invariant by holonomy. In particular, if the leaf space of the foliation is a bona fide manifold, then the geometric structure descends as a geometric structure on the leaf space. In this paper, we are concerned with {\slshape transversely CR real foliations}, hence the normal bundle is endowed with a CR structure invariant by holonomy.

\begin{lemma}
	\label{lemmaCR}
	Let $T\in\mathbb{V}\setminus\I$. Then $T$ defines on $G$ a left invariant transversely CR real foliation of dimension $\dim (T\cap\bar T)$.
\end{lemma}

\begin{proof}
	Let $T\in\mathbb{V}\setminus\I$. Then $T\cap\bar T$ is a positive-dimensional complex subspace of $\mathfrak{G}_\C$ invariant by conjugation, hence is the complexification of some real subspace $T\mathscr{F}$ of $\mathfrak{G}$. Moreover,  $T\cap\bar T$ and thus $T\mathscr{F}$ are preserved by Lie bracket. By Frobenius Theorem, the corresponding left invariant distribution, that we still denote by $T\mathscr{F}$ is tangent to a real foliation $\mathscr{F}$ of $G$. It is itself left invariant, that is left translations preserve the leaves. Setting $E:=T/(T\cap\bar T)$, we see that $E\cap\bar E$ is reduced to zero. Then $E\oplus \bar E$ is a complex subspace of the complexification of the normal bundle to the foliation $N\mathscr{F}:=TG/T\mathscr{F}$ that is invariant by complex conjugation. Hence it is the complexification of the real subspace
	\begin{equation}
		\label{E}
		E_\R:=\{v+\bar v\mid v\in E\}
	\end{equation}
	of $N\mathscr{F}$. To sum up, we have a decomposition
	\begin{equation}
		\label{CRnormal}
		E_\R\otimes_\R \C=E\oplus \bar E\subset N\mathscr{F}\otimes_\R \C
	\end{equation}
	But this is exactly the definition of an almost-CR structure on $N\mathscr{F}$. We derive from the involutivity of $T$ that $E$ is involutive as well, so this almost-CR structure is integrable, i.e. is a CR structure.
	
	Finally, holonomy morphisms of $\mathscr{F}$ are left translations. Since $E_\R$ and $E$ are left invariant from the construction, this CR structure of $N\mathscr{F}$ is preserved by holonomy and we are done. 
\end{proof}

\begin{remark}
	\label{rktCR}
	The inclusion in \eqref{CRnormal} is always strict. Hence the foliation $\mathscr{F}$ is not transversely holomorphic. Still, transversely CR foliations are close to  transversely holomorphic foliations and to the Polarized CR structures of \cite{Polarized}. 
\end{remark}

There is an alternative geometric interpretation of elements of $\mathbb{V}\setminus \I$. Let $G_\C$ be the connected and simply-connected complex Lie group associated to the Lie algebra $\mathfrak{G}_\C$.

\begin{lemma}
	\label{lemmaHF}
	Let $T\in\mathbb{V}\setminus\I$. Then $T$ defines on $G_\C$ a left invariant holomorphic foliation of complex dimension $\dim T$.
\end{lemma}

\begin{proof}
	The involutive subspace $T$ of $\mathfrak{G}_\C$ corresponds to a left invariant involutive holomorphic distribution of $G_\C$, which is thus tangent to a holomorphic foliation of dimension $T$.
\end{proof}
\section{Compactification of the Moduli Stack}
\label{secmodulistack}

The moduli stack of left invariant complex structures on $G$ is the category fibered in groupoids over the analytic site $\mathfrak{A}$ whose objects are families over a base $B\in\mathfrak{A}$ given by smooth and proper morphisms with all fibers isomorphic to $G$ endowed with some left invariant complex structure and whose morphisms over $B\to B'$ are morphisms of a family over $B$ to a family over $B'$. It is thus a subcategory of the moduli stack of complex structures on $G$ as defined in \cite{LMStacks}. We denote it by $\mathscr{M}_\ell(G)$. More generally, we define in the same way the stack $\mathscr{M}_\ell(M)$ where $M=G/\Gamma$ for some subgroup $\Gamma$ of $G$ acting on the right on $G$ hence commuting with left translations. Any left invariant complex structure on $G$ descends thus as a left invariant complex structure on the $C^\infty$ homogeneous space $M$ but the corresponding moduli stacks are in general different since the collection of isomorphisms that may exist between distinct structures depends in general on $\Gamma$, see \S \ref{secexamples}.

There exists a tautological family of left invariant structures over $\I$. Just consider on the product $G\times\I$ the complex structure on $G\times\{T\}$ given by $T$ as explained in Lemma \ref{lemmagrass}. We denote it by $\mathscr{U}_\ell\to\I$. This family defines a mapping $\I\to\mathscr{M}_\ell(M)$ by Yoneda's Lemma. This mapping is smooth and surjective. In fact, the fiber product $\I\times_{\kern-1pt \mathscr{M}_\ell(M)}\I$ describes the isomorphisms between left invariant complex structures 
\begin{equation}
	\label{iso}
	\begin{tikzcd}
		(M,T')\arrow[r,"f"',"\simeq"]&(M,T)
	\end{tikzcd}
\end{equation}
and thus, noting that \eqref{iso} forces $T'$ to be equal to $df^{-1}(T)$, is equal to the analytic space
\begin{equation}
	\label{fiberproduct}
	\I\times_{\kern-1pt\mathscr{M}_\ell(M)}\I=\{(T,f)\in\text{Diff}(M,\mathscr{U}_\ell)\mid df^{-1}(T)\in\I\}
\end{equation}
where the set $\text{Diff}(M,\mathscr{U}_\ell)$ is the set of diffeomorphisms from $M$ to a fiber of the family $\mathscr{U}_\ell\to\I$, see \cite{Douady}. Now, the projection of \eqref{fiberproduct} to $\I$ is smooth and surjective, cf. \cite{LMStacks}. Hence, $\I\to \mathscr{M}_\ell(M)$ is an atlas for $\mathscr{M}_\ell(M)$ which is thus an analytic stack.

\begin{remark}
	\label{rkfamilies}
	This abstract construction expresses in a more detailed and categorical way the following facts. Let $B$ be an analytic space. A family of left invariant structures on $M$ above $B$ is given by gluing over some open cover $(B_\alpha)$ of $B$ pull-backs of $\mathscr{U}_\ell\to \I$ by holomorphic maps $f_\alpha:B_\alpha\to\I$. Gluings are done following a holomorphic $1$-cocycle $(f_{\alpha\beta})$ with values in \eqref{fiberproduct}. Morphisms between families can also be described over some open cover of their bases by some compatible collection of holomorphic maps with values in \eqref{fiberproduct}. As a consequence, the category $\mathscr{M}_\ell(M)$ can be completely recovered from the date of $\mathscr{U}_\ell\to\I$ and \eqref{fiberproduct}.
\end{remark}

But we may easily take a step further and enlarge our notion of family. Define a family $\mathscr{X}\to B$ as a locally trivial $C^\infty$ bundle over $B$ with fiber $M$ plus a $n$-dimensional left invariant involutive subundle $E$ of $T_F\mathscr{X}\otimes_\R\C$, the complexified tangent bundle to the fibers of $\mathscr{X}\to B$. Let $B^*\subset B$ be the points $x\in B$ above which the fiber $E_x$ of $E$ satisfies $E_x\cap\bar E_x=\{0\}$. Over this possibly empty open set, the restriction $\mathscr{X}^*\to B^*$ of $\mathscr{X}$ is a family of left invariant complex structures on $M$ as before, that is an object of $\mathscr{M}_\ell(M)$. Define a morphism between two such families $\mathscr{X}\to B$ and $\mathscr{X}'\to B'$ as a bundle map
\begin{equation*}
	\begin{tikzcd}
		\mathscr{X} \arrow[d]\arrow[r,"F"]&\mathscr{X}'\arrow[d]\\
		B\arrow[r,"f"']&B'
	\end{tikzcd}
\end{equation*}
with $f$ holomorphic and the differential of $F$ inducing a bundle morphism $T_F\mathscr{X}\otimes_\R\C\to T_F\mathscr{X}'\otimes_\R\C$ that sends $E$ to $E'$. The restriction of such a morphism to $\mathscr{X}^*\to B^*$ is a morphism of $\mathscr{M}_\ell(M)$.

Our set of morphisms is however too big, for the automorphism group of a transversely CR real foliation may be infinite-dimensional even if $G$ is compact, cf. \S \ref{subsectori}. So we restrict to the morphisms as above that moreover commute with left translations. We denote by $\bar{\mathscr{M}}_\ell(M)$ the stack over $\mathfrak{A}$ formed by the previously defined generalized notions of families and of morphisms of families (in the restricted sense).

In general, there is no inclusion of $\mathscr{M}_\ell(M)$ into $\bar{\mathscr{M}}_\ell(M)$, because it is not clear that holomorphic maps commute with left translations. To overcome this problem, we introduce the condition
\begin{hypothesis}
	\label{condition}
	Let $J$ and $J'$ be two left invariant complex structures on $M$. Let $f$ be a biholomorphism between $(M,J)$ and $(M,J')$. Then $f$ commutes with left translations on $M$.
\end{hypothesis}

Under hypothesis \ref{condition}, we have a natural inclusion of $\mathscr{M}_\ell(M)$ into $\bar{\mathscr{M}}_\ell(M)$. Moreover, the tautological family $\mathscr{U}_\ell\to \I$ extends to a family $\mathscr{U}\to\mathbb{V}$ with $\mathscr{U}^*\simeq\mathscr{U}_\ell$ and $\mathbb{V}^*=\I$. 

We note that every biholomorphism between compact parallelizable manifolds satisfies this additional condition. Indeed, let $f$ be a biholomorphism between $M$ and $M'$ two compact parallelizable manifolds. Since the tangent bundles of $M$ and $M'$ are holomorphically trivialized by left invariant vector fields, $f$ sends a left invariant vector field of $M$ to a linear combination of left invariant vector fields of $M'$. By compacity, the coefficients are constant. So $f$ commutes with left translations. This is still true in many non parallelizable examples, cf. \S \ref{subsecHopf}. 

The stack $\bar{\mathscr{M}}_\ell(M)$ is not an analytic stack since the automorphism group of a transversely CR real foliation is not in general a {\slshape complex} Lie group, see \S \ref{subsectori}. However, the family $\mathscr{U}\to\mathbb{V}$ defines a surjective mapping $\mathbb{V}\to\bar{\mathscr{M}}_\ell(M)$, which is smooth in following CR sense. The corresponding fiber product 
\begin{equation}
	\label{fiberproduct2}
	\mathbb{V}\times_{\kern-1pt\bar{\mathscr{M}}_\ell(M)}\mathbb{V}=\{(T,f)\in\text{Diff}_\ell(M,\mathscr{U})\mid df^{-1}(T)\in\mathbb{V}\}
\end{equation}
where the subscript $\ell$ in $\text{Diff}_\ell(M,\mathscr{U})$ means that we restrict to diffeomorphisms commuting with left translations, can be endowed with a CR structure which satisfies above any sufficiently small open set $U$ of $\mathbb{V}$ the following diagram
\begin{equation*}
	\begin{tikzcd}
		p^{-1}(U) \arrow[dr, "p"']\arrow[rr,"\simeq"]&&U\times V\arrow[dl,"proj."]\\
		&U
	\end{tikzcd}
\end{equation*}
where $p$ is the projection map from \eqref{fiberproduct2} to $\mathbb{V}$, the set $V$ is an open set of some euclidean space $\R^N$ and we endow $p^{-1}(U)$ with the Levi flat CR srtucture that makes of the isomorphism on the top a CR isomorphism. This induced CR structure does not correspond to the natural complex structure of \eqref{fiberproduct} when restricted to $\I$. It purely comes from the complex structure of $\mathbb{V}$ that is pulled-back through $p$. 

Since $\mathbb{V}$ is compact, $\bar{\mathscr{M}}_\ell(M)$ is compact and we may state

\begin{theorem}
	\label{thmcompactification}
	Assume Hypothesis \ref{condition}. Then, the natural inclusion map $\mathscr{M}_\ell(M)\hookrightarrow\bar{\mathscr{M}}_\ell(M)$ exhibits $\bar{\mathscr{M}}_\ell(M)$ as a compactification of $\mathscr{M}_\ell(M)$.
\end{theorem}

\section{Examples}
\label{secexamples}
\subsection{Elliptic curves}
\label{subsectori}
Let us begin with $G=\mathbb{S}^1\times\mathbb{S}^1$. Then $\mathfrak{G}_\C$ is $\C^2$ with null bracket and $\mathbb{V}$ is the complex projective line $\mathbb{P}^1(\mathfrak{G}_\C)$. More precisely, points $[\tau : -1]$ with $\tau\in\mathbb{H}$ encode the elliptic curve $\mathbb{E}_\tau$, whereas points $[\bar\tau : -1]$ encode the elliptic curve $\mathbb{E}_{\bar\tau}$ which is isomorphic to $\mathbb{E}_\tau$ but induces on $G$ the opposite orientation to that induced by $\mathbb{E}_\tau$, cf. Remark \ref{rkoriented}. Finally points $[a:b]$ with $a$ and $b$ real correspond to $G$ endowed with a real linear foliation $\mathscr{F}_s$ of slope $s:=ba^{-1}\in\R\cup\{+\infty\}$. Observe that the normal bundle to the foliation has real dimension $1$, so the transverse CR structure is a zero-dimensional subbundle. Hence the tautological family is the universal family over $\mathbb{H}$ and over $-\mathbb{H}$, and the family of foliations $\mathscr{F}_s$ above $\R\mathbb{P}^1$. 

Every left invariant complex structure defines on $G$ a structure of a complex Lie group. Since $G$ is moreover compact, Hypothesis \ref{condition} is satisfied and Theorem \ref{thmcompactification} applies. The connected component of the identity of the automorphism group of $\mathscr{F}_s$ is the {\slshape real} Lie group of translations of $G$, whereas of course that of $\mathbb{E}_\tau$ is complex and equal to $\mathbb{E}_\tau$. Here it is crucial to stick to the definition of automorphisms of $\mathscr{F}_s$ that commute with translations otherwise we would have infinite dimensional automorphism groups. Indeed, in the case of a rational $s$, the leaf space of $\mathscr{F}_s$ is a circle and the set of diffeomorphisms of $G$ that preserve $\mathscr{F}_s$ projects surjectively on the full group of $C^\infty$ diffeomorphism of the circle that serves as leaf space. 

Going back to the left invariant structures on $G$, we may now construct a family of translation groups over $\mathbb{P}^1$ corresponding to the tautological family, i.e. the fiber over $\tau$ is the translation group of the fiber over $\tau$ of the tautological family. This family $\mathscr{T}$ of translations is the universal family of elliptic curves compactified by a family of copies of $G$ above $\R\mathbb{P}^1$. In other words, consider the product $G\times\mathbb{P}^1\to \mathbb{P}^1$ and endow its restriction to $\mathbb{H}\cup -\mathbb{H}$ with the complex structure of the universal family. This is $\mathscr{T}$. 

\begin{remark}
	\label{rkCR}
	The CR structure on $\mathscr{T}$ alluded to in \S \ref{secmodulistack} is just the trivial one given by the product $G\times\mathbb{V}$. 
\end{remark}
Finally $\text{GL}_2(\mathbb{Z})$ acts on $\mathbb{V}$, on the tautological family and on the family of translations. The stack  $\bar{\mathscr{M}}_\ell(G)$ is the quotient stack $[\mathscr{T}/\text{GL}_2(\mathbb{Z})]$ and its geometric quotient is $\mathbb{V}/\text{GL}_2(\mathbb{Z})$. The isotropy group of $\mathscr{F}_s$ is countable in the special cases $s=0$ and $s=+\infty$ so this is not an orbifold.
 
We note that $G_\C$ is equal to $\mathfrak{G}_\C$ and the left invariant foliations of Lemma \ref{lemmaHF} are just the linear foliations of $\mathfrak{G}_\C\simeq\C^2$ by parallel lines. 

\begin{remark}
	\label{rkKontsevich}
	Alternatively, we may think of $\mathscr{F}_s$ with $s$ irrational as a non-commutative torus, so this construction exhibits non-commutative tori as limits of elliptic curves, as suggested in \cite[\S 1.39]{KE}. The moduli stack we obtain, namely $[\mathscr{T}/\text{GL}_2(\mathbb{Z})]$, is compatible with this interpretation since non-commutative tori are Morita equivalent if and only if the slopes are related through an element of $\text{GL}_2(\mathbb Z)$ by \cite{Rieffel}.
\end{remark} 
  
\subsection{Higher-dimensional Tori}
\label{subsecTori}
The picture for $G=(\mathbb S^1)^{2n}$ is similar to that described in \S \ref{subsectori}. Here $\mathbb{V}$ is the complex grassmannian $\text{Gr}_n(\C^{2n})$. For $k$ between $0$ and $n$, define a $k$-point in $\text{Gr}_n(\C^{2n})$ as a complex $n$-plane of $\C^{2n}$ that intersects $\R^{2n}$ in a subspace of dimension at least $k$. Let $\mathbb{V}_k$ define the set of $k$-points. Then $\mathbb{V}_0$ is the whole $\mathbb{V}$ whereas $\mathbb{V}_n$ is the real grassmannian $\text{Gr}_n(\R^{2n})$ included in $\mathbb{V}$ through the inclusion $\R^{2n}\subset\C^{2n}$. Each Schubert cell of $\text{Gr}_n(\C^{2n})$ admits a stratification by $k$-points. Notice that points in $\mathbb{V}_0\setminus\mathbb{V}_1$ encode complex tori regardless of the orientation, cf. Remark \ref{rkoriented}. And $k$-points of $\mathbb{V}_k\setminus\mathbb{V}_{k+1}$ correspond to $G$ endowed with a real linear $k$-dimensional foliation $\mathscr{F}_s$ and a transverse CR structure tangent to a transverse complex linear $n-k$-dimensional foliation. Once again, $n$-points are just linear real $n$-foliation with no transverse CR-structure for dimension reason.

Diffeomorphisms that commute with left translations lift as affine transformations of $\R^{2n}$. Hence, from the one hand, Hypothesis \ref{condition} is satisfied and Theorem \ref{thmcompactification} applies. From the other hand, the automorphism group of the structure induced by a $k$-point is the subgroup of the affine group of $\R^{2n}$ that preserve both the real and the CR linear foliations. In particular, it is a finite-dimensional connected real Lie group. 

The construction of the compactified moduli stack of higher-dimensional tori goes along the same lines as that of elliptic curves. Start with the product $G\times \text{Gr}_n(\C^{2n})$. Endow its restriction to $\mathbb{V}_0\setminus\mathbb{V}_1$ with a structure of a universal family of tori. This gives a smooth mapping from $\mathbb{V}_0\setminus\mathbb{V}_1$ to the unoriented\footnote{ that is with no fixed orientation on $G$, cf. \cite[Remark 2.1]{LMStacks}.} Teichmüller stack making it an atlas of it. Calling $\mathscr{T}$ the space $G\times \text{Gr}_n(\C^{2n})$ with its complex structure on $\mathbb{V}_0\setminus\mathbb{V}_1$, then $\text{GL}_{2n}(\mathbb{Z})$ act on $\mathscr{T}$ and the stack  $\bar{\mathscr{M}}_\ell(G)$ is the quotient stack $[\mathscr{T}/\text{GL}_{2n}(\mathbb{Z})]$. Its geometric quotient is $\mathbb{V}/\text{GL}_{2n}(\mathbb{Z})$. Once again, $G_\C$ is equal to $\mathfrak{G}_\C$ and the left invariant foliations of Lemma \ref{lemmaHF} are just the linear foliations of $\mathfrak{G}_\C\simeq\C^{2n}$ by parallel $n$-planes.

\subsection{Left-invariant Hopf Surfaces}
\label{subsecHopf}
We take now $G=\text{SU}_2\times \mathbb{S}^1$. The set of left invariant complex structures on $G$ has been investigated in \cite{Sasaki}. We review its results. In the Lie algebra $\mathfrak{s}\mathfrak{u}_2$, consider the basis
\begin{equation*}
	X=\begin{pmatrix}
		i &0\\
		0 &-i
	\end{pmatrix},
	\qquad
	Y_1=\begin{pmatrix}
		0 &1\\
		-1 &0
	\end{pmatrix},
	\qquad
	Y_2=\begin{pmatrix}
		0 &i\\
		i &0
	\end{pmatrix}
\end{equation*} 
Then a basis of $\mathfrak{G}=\mathfrak{s}\mathfrak{u}_2\times\R$ is given by 
\begin{equation}
	T=(0,1),\qquad S=(X,0),\qquad A=(Y_1,0),\qquad B=(Y_2,0)
\end{equation}
The set $\mathscr{I}_\ell$ maps through \eqref{Imapping} to the subset of $\text{Gr}_2(\mathfrak{G}_\C)$ containing
\begin{enumerate}[(I)]
	\item $\langle T-\tau S,A+iB\rangle$ and its conjugate ($\tau$ is a complex number with non-zero imaginary part)
	\item $\langle 2S+a(A+iB)+a^{-1}(A-iB),bT +ia(A+iB)-ia^{-1}(A-iB)\rangle$ ($a\in\C^*$ and $b$ is a complex number with non-zero real part)
\end{enumerate}
and $\mathbb{V}$ is its closure in $\text{Gr}_2(\mathfrak{G}_\C)$. It is isomorphic to $\mathbb{P}^1\times\mathbb{P}^1$. More precisely, the set of complex structures of type (II) is parametrized by $(a,b)\in\C^*\times\C\setminus i\mathbb{R}$ with natural compactification $\mathbb{P}^1\times\mathbb{P}^1$. On the other hand, both the set of complex structures of type (I) and its conjugate compactify as a $\mathbb{P}^1$ with parameter $\tau$. A straightforward computation shows that the map $(0,b)\mapsto \tau=-2ib^{-1}$, resp. $(+\infty,b)\mapsto \tau=2ib^{-1}$, identifies the compactified type (I) set with $\{0\}\times \mathbb{P}^1$, resp. its conjugate with $\{+\infty\}\times\mathbb{P}^1$.

A complex structure of type (I) with $\tau\in\mathbb{H}\cup -\mathbb{H}$ is an elliptic Hopf surface fibering over $\mathbb{P}^1$ with fiber $\mathbb{E}_\tau$ and is isomorphic to the quotient of $\mathbb{C}^2\setminus\{0\}$ by the group generated by the homothety $e^{2i\pi\tau}Id$. Its conjugate gives the same Hopf surface but with the opposite orientation on the base. The transversely CR foliations correspond to the cases $\tau\in\R\mathbb{P}^1$. Here the elliptic fibration becomes a fibration over $\mathbb{P}^1$ with fibers $\mathbb{S}^1\times\mathbb{S}^1$ endowed with a linear foliation as in \S \ref{subsectori}. The CR structure on the fibration is given by the pull-back of the complex structure of $\mathbb{P}^1$. 

Now, a complex structure of type (II) is isomorphic to a complex structure of type (I), cf. \cite{Sasaki}, so that it is enough to restrict 
$\mathbb{V}$ to the projective line 
\begin{equation*}
	\mathbb{V}_0={\{\langle T-\tau S,A+iB\rangle\mid\tau\in\mathbb{P}^1\}}
\end{equation*}
Set $\lambda:=e^{2i\pi\tau}$ and denote by $H_\lambda$ the elliptic Hopf surface $\mathbb{C}^2\setminus\{0\}/\langle \lambda Id\rangle$. It is easy to check that different values of $\lambda$ give non-isomorphic Hopf surfaces and that the automorphism group of $H_\lambda$ is $\text{GL}_2(\C)/\langle \lambda Id\rangle$. Hence Hypothesis  \ref{condition} is satisfied and Theorem \ref{thmcompactification} applies. From all that, we may infer a presentation of $\bar{\mathscr{M}}_\ell(G)$ as a quotient stack in the same spirit of \S \ref{subsectori}. Consider the family
\begin{equation}
	\label{family}
	\text{GL}_2(\C)\times (\mathbb{H}\cup -\mathbb{H})/_{\displaystyle\langle e^{2i\pi\tau}Id,Id\rangle}\longrightarrow \mathbb{H}\cup -\mathbb{H}
\end{equation}
and observe that all complex manifolds
\begin{equation*}
	K_\tau:=\text{GL}_2(\C)/_{\displaystyle\langle e^{2i\pi\tau}Id\rangle}
\end{equation*}
are complex Lie groups all isomorphic as {\slshape real} Lie groups. Hence, denoting by $K$ the common underlying $C^\infty$ Lie group, we extend the family \eqref{family} over the whole $\mathbb{P}^1$ by adding a copy of $K$ above all points of $\mathbb{P}^1\setminus (\mathbb{H}\cup -\mathbb{H})$. That is, we consider $K\times \mathbb{P}^1\to \mathbb{P}^1$ with its restriction to $\mathbb{H}\cup -\mathbb{H}$ endowed with the complex structure of \eqref{family}. Denote by $\mathscr{A}$ this family. Finally, we just have to identify the fiber over $\tau$ with the fiber over $\tau+1$. So we let $1\in\mathbb{Z}$ act on $K\times \mathbb{P}^1$ as the identity on $K$ and the translation by $1$ on $\mathbb{P}^1$. We note that this action is holomorphic on \eqref{family}. The stack  $\bar{\mathscr{M}}_\ell(G)$ is the quotient stack $[\mathscr{A}/\mathbb{Z}]$ and its geometric quotient is $\mathbb{P}^1/\mathbb{Z}$. The isotropy group of $+\infty$ is countable so once again this is not an orbifold. 

Finally, $G_\C$ is equal here to $\text{SL}_2(\C)\times\C$ and the left invariant complex structures on $G$ can thus be realized as left invariant holomorphic foliations of $\text{SL}_2(\C)\times\C$.

\end{document}